\documentclass[10pt]{article}
   \usepackage{graphicx}
\usepackage{latexsym}
\usepackage[utf8x]{inputenc}
\usepackage{float}
\usepackage{amsmath}
\usepackage{cite}
\usepackage{amssymb}
\usepackage{amsthm}
\usepackage{url}
\usepackage{tabulary}
\usepackage{booktabs}
 \usepackage{helvet}

\theoremstyle{plain}
\newtheorem{theorem}{Theorem}[section]
\newtheorem{pro}[theorem]{Proposition}
\newtheorem{lemma}[theorem]{Lemma}

\theoremstyle{definition}
\newtheorem*{defi}{Definition}
\newtheorem{Ex}[theorem]{Example}

\theoremstyle{remark}

\def \Q {{\mathbb Q}}

\def \W {{\mathbb W}}
\def \Z {{\mathbb Z}}

\def \0 {{\mathbf 0}}
    
 \title{Bhargava factorials and irreducibility of
integer-valued polynomials}
  
\author{Devendra Prasad\\
 Department of Mathematics\\ 
IISER-Tirupati\\ Tirupati, Andhra Pradesh\\
 India, 517507\\
 devendraprasad@iisertirupati.ac.in}

\begin{document}
        \date{}
     \maketitle

     \begin{abstract}
     
    The ring of integer-valued polynomials  over a given subset $S$ of   $\Z$ (or $ \mathrm{Int}(S,\Z ))$ is defined as the set of polynomials in $\Q[x]$ which maps $S$  to $\Z$. In factorization theory, it is  crucial to check the irreducibility of a polynomial. In this article, we make Bhargava factorials   our main tool  to check the irreducibility of a given polynomial 
 $f \in   \mathrm{Int}(S,\Z ))$. We also generalize our results to arbitrary subsets of a Dedekind domain.

     \end{abstract}

 \section{Introduction}
     
     In his celebrated work, Bhargava \cite{Bhar1} (see \cite{Bhar3} also) generalized the notion of factorials to an arbitrary subset of $\Z$ (or a Dedekind domain).  These factorials are intrinsic to the given subset. He answered several open questions and generalized some mathematical identities by using these factorials. He obtained these factorials by the notion of $p$-orderings. For the sake of completeness, we recall all these concepts. 
     
     \medskip
     
    Let $S$ be an arbitrary subset of  $\mathbb{Z}$ and  $p$ be  a fixed prime. A $p$-ordering of $S$ is a sequence  $a_0, a_1, a_2,\cdots $ of elements of $S$ that is formed as follows:
    
     \begin{itemize}

  \item Choose any element $a_0 \in S;$
 \item  Choose an element $a_1 \in S$ that minimizes the highest power of $p$ dividing $a_1 -a_0;$
 \item  Choose an element $a_2 \in S$ that minimizes the highest power of $p$ dividing $(a_2 -a_0)(a_2-a_1);$ 

In a similar way,
 
\item  

Choose an element $a_k \in S$ that minimizes the highest power of $p$ dividing $\prod_{i=0}^{k-1} (a_k-a_i).$
  
         \end{itemize}
         

  %
 
  \medskip
  
Corresponding to a given $p$-ordering, we also get a sequence of  powers of $p$,  namely, the powers of $p$ minimized at each step, which is termed as {\em $p$-sequence}. Bhargava proved that this sequence does not depend on any choice of $p$-ordering.
If the sequence $\{ a_i \}_{ i \geq 0}$  is a $p$-ordering of $S\ \forall\ p \in \Z,$ then we say that $\{ a_i \}_{ i \geq 0}$  is a {\em simultaneous $p$-ordering.}  If a  simultaneous $p$-ordering contains $k$ terms, then we say that  the sequence forms a {\em simultaneous $p$-ordering  of length $k$}. 

\medskip

  For a given integer $d$ and a given prime $p$, we denote the highest power of $p$ dividing $d$ by $w_p(d).$ For example, $w_2(12)= 2^2$. With this notation, the generalized factorial of index $k\ \forall\ k \geq 0$ is defined as
  
  $$k!_S= \prod_p w_{p}((a_k-a_0)(a_k-a_1) \ldots (a_k-a_{k-1})).$$

\medskip

The number of primes in the above definition is always finite. 
Recall that, the ring of integer-valued polynomials over a subset $S  \subseteq \Z $ is defined as

$$ \mathrm{Int}(S,\Z  )=    \{f \in \Q[x]: f(S) \subseteq \Z \} . $$

Denote the set of polynomials of  $\mathrm{Int}(S,\Z  )$ of degree $k$ by $\mathrm{Int}_k (S,\Z  ).$
It turns out that
\begin{equation}\label{eq:factorial}
  k!_S = \gcd \{a : a \mathrm{Int}_k (S,\Z  ) \subseteq \Z[x]  \}. 
  \end{equation}
\medskip

The ring of integer-valued polynomials can be defined for any  subset of a   domain. This ring has emerged as a broad area of research in the previous few decades. In ring theory, irreducibility is one of the most exciting and well-studied concepts with a venerable history. However, for the ring $\mathrm{Int}(S,\Z  )$, irreducibility is not studied that much.  
For some criteria on  irreducibility, we refer to Prasad, Rajkumar and Reddy \cite{prasadsurvey}, where readers can find a whole section on the irreducibility of integer-valued polynomials.  

\medskip

This article is an attempt to understand the irreducibility of a polynomial  $f \in \mathrm{Int}(S,\Z  )$ by the notion of generalized factorials (or Bhargava factorials). In this article, we get a new criterion to test the reducibility of  polynomials in $\mathrm{Int}(S,\Z  )$  and then we generalize this result to the case of an arbitrary subset of a Dedekind domain.

\medskip

 The summary of the paper is as follows. In Section \ref{dkord}, we define the notion of $d_k$-orderings with some examples.    Section \ref{sec irr}  presents our main result on the irreducibility of polynomials in $\mathrm{Int}(S,\Z  )$ . We explicitly give some examples to explain our theorem and show how our results can be more viable in the study of integer-valued polynomials. 

\section{$d_k$-orderings}\label{dkord}     We start this section with the definition of $d_k$-orderings which is a cornerstone of our work.
  \begin{defi} For    given integers $d $ and $k$,     let $p_1, p_2,\ldots, p_r$
 be all the  prime    divisors of $d$. For   $1 \leq j \leq r$, let  $\{ u_{ij} \}_{i \geq 0} $ be a $p_j$-ordering of $S \subseteq \Z.$ 
 Then a $d_k$-ordering $ \{ x_i \}_{0 \leq  i \leq k}$   of $S$     is a solution to the following congruences
 
   
   \begin{equation}\label{crt}
   x_i \equiv u_{ij} \pmod{ p_j^{e_{kj}+1}}\ \forall\ 1 \leq j \leq r,
   \end{equation}

  where $p_j^{e_{kj} }=w_{p_j}(k!_S)$. 
\end{defi}

If a $d_k$-ordering of a subset $S \subseteq \Z$ belongs to $S$, then it is a $p$-ordering 
 for all primes $p$ dividing $d$. However,    a $d_k$-ordering   gives  the    same $p$-sequences for all $p$ dividing $d$, even if  $d_k$-ordering doesn't belong to $S$; this holds because by construction  $\prod_{
k=0,...,k−1}(u_{kj}-
u_{ij} )$ and  $\prod_{
k=0,...,k−1}
(x_k - x_i)$ have the same $p$-adic valuations for all $p$ dividing $d$.  
In the case when $d=0,$ a $ d_k$-ordering is clearly a simultaneous $p$-ordering of length $k+1.$  Before proceeding, we give some examples of $d_k$-orderings. 
Here (and throughout the article) $\W$ denotes the set $ \{0,1,2,\ldots \}.$

 \begin{Ex}  Let $S =\{ 0^2, 1^2, 2^2, \ldots  \}$. 
 Then,  $0^2, 1^2, 2^2, \ldots,  k^2 $ 
 is a  $d_k$-ordering for every $(d,k) \in \W^2$.
 
\end{Ex} 
 \begin{Ex}  Let $S =a \Z+b$  where $a,b \in \Z.$ Then the sequence $b,a+b,\ldots, ak+b$ is  a $d_k$-ordering  for every  $(d,k) \in \W^2$. 
  Actually, any $k +1$ consecutive term  of the set $S$ forms a  $d_k$-ordering for every $(d,k) \in \W^2$.
 
\end{Ex}
   
  For a given subset $S \subseteq \Z,$ the {\em fixed divisor } of a polynomial $f \in \Z[x]$ over $S$ is denoted by $d(S,f)$ and is defined as follows 
  
  $$d(S,f)=  \gcd \{ f(a): a \in S    \} .$$
  
 We refer to  Prasad, Rajkumar and Reddy \cite{prasadsurvey} for a brief survey on the topic. The readers may find some exciting results and  applications of fixed divisors in  \cite{Peruginelli},   \cite{Devendrafixed},  \cite{Prasad2019},  \cite{Devendra} etc.

  \medskip

  Let  $\{ a_i \}_{i \geq 0}$ be a sequence of distinct elements  of $S.$ Suppose that  for every $k >0,\ \exists\ l_k \in \Z,$  such that for every polynomial $f$ of degree $k$  $$d(S,f)=(f( a_0),f(a_1), \ldots, f(a_{l_k})),$$
   and no proper subset of $\{ a_0, a_1, \ldots,  a_{l_k} \}$ 
  determines the fixed divisor of all the degree $k$ polynomials. Then we say that the sequence $\{ a_i \}_{i \geq 0}$ is   a {\em  fixed divisor sequence} (see Prasad, Rajkumar and Reddy \cite{prasadsurvey}).
   For instance, the sequence $0,1,2, \ldots$ is a fixed divisor sequence  in $\Z$ with $l_k=k\ \forall\ k >0.$

  \begin{Ex} Let $S $ be a subset of $\Z$ with a fixed divisor sequence $\{ a_i \}_{i \geq 0}$. Assume $ l_k=k\ \forall\ k \geq 0,$ then $a_0,a_1, \ldots, a_k$ is a $d_k$-ordering of length $k$ for every $(d,k) \in \W^2 $.

  \end{Ex}

\medskip

Let  $(d,p,i) \in \W^3$ be a given triplet with $w_p(i!_S) \leq  w_p(d)$ and 
  $ S \subseteq \Z$ be a given subset. 
Assume  $\mu_i (d , p)$ denote the    power of  the prime $p$ such that 
$$ \mu_i (d , p)  w_p(i!_S)=w_p(d).$$

The sequence $(\mu_i (d , p))_{i \geq 0}$ is a decreasing sequence in the powers of $p$. 
 The function $ \mu_i (d , p)$ also depends on the set chosen (since  generalized factorials depend).  For instance, if $(d,k)=(6,4)$ and $S=\Z$ then
$$\mu_0(6 , 3)=3^1\ \mathrm{and}\ \mu_1(6 , 3)=3^1. $$

However, if $S=3\Z$, then
$$\mu_0(6 , 3)=3^1\ \mathrm{and}\ \mu_1(6 , 3)=3^0. $$
Hence, we always assume that, in the notation $\mu_i (d , p)$, 
  the subset automatically comes   from the context. 

\medskip

   In some special cases (for instance, when $S=\Z$), we can get a precise formula for the sequence $\{ \mu_i (d , p) \}_{ i \geq 0}$ very easily.
\medskip

The following Lemma connects generalized factorials and $d_k$-orderings.

\begin{lemma}\label{conlem}   Let $a_0, a_1, \ldots, a_k$ be a $d_k$-ordering of $S \subseteq \Z$ for  
$(d, k) \in \W^2$.  Then for any prime $p$ dividing $d$

$$w_p(d(S,F_r))= w_p(r!_S)\ \forall\ 0 \leq  r\ \leq k, $$
where $F_r=(x-a_0)(x-a_1) \ldots (x-a_{r-1}).$
\end{lemma}

 \begin{proof}  Observe that 
$ w_p(d(S,F_r))=  w_p( (a_{r}-a_0)(a_{r}-a_1) \ldots (a_{r}-a_{r-1})).$  Assume $b_0,b_1, \ldots$ be the   $p$ ordering  such that

 \begin{equation*} 
   a_i \equiv b_{i} \mod p^{e_{k}+1}\ \forall\ 1 \leq i \leq k,
   \end{equation*}

where $p^{e_{k} }=w_{p}(k!_S)$. Clearly   $ w_p( (a_{r}-a_0)(a_{r}-a_1) \ldots (a_{r}-a_{r-1}))$ is same as $   w_p( (b_{r}-b_0)(b_{r}-b_1) \ldots (b_{r}-b_{r-1}))$ and which is, in fact, $  w_p(r!_S) . $

\end{proof} 

 \medskip
 
  We end this section with the following proposition whose proof follows by our way of construction of $d_k$-orderings.
 
  \begin{pro}\label{fdcr} Let $a_0, a_1, \ldots, a_k$ be a $d_k$-ordering of $S \subseteq  \Z $ for given integers $d$ and $k$. 
   Then, for any  polynomial $ f \in \Z[x]$  of degree $k' \leq k$  
   with $d(S,f) \mid d$, we have
  
    $$d(S,f)  = (   f( a_0), f(  a_1), \ldots, f( a_{k'}    ) ).$$
  \end{pro}

     \section{Irreducibility of polynomials in $ \mathrm{Int}(S, \Z)$}\label{sec irr}
 Before coming to our main results, we fix some notations and assumptions for the whole section. Throughout this section, $S$ denotes an arbitrary subset of $\Z$  and $a_0, a_1, \ldots, a_k$ denotes a $d_k$-ordering of $S$, where $d$ and $k$ automatically come from the context. For a given positive integer $i \leq k+1,\ F_i$ denotes $(x-a_0)(x-a_1) \ldots (x-a_{i-1}).$  Whenever a polynomial $f \in \Q[x] $ is expressed as $\tfrac{ g  }{d}, $ we assume that both $g \in \Z[x]$ and $d \in \Z$ are unique.
 
 \medskip
 
Throughout the section, for a given polynomial $f$, whenever  we   say $f$ is irreducible, then 
 the subset automatically come from the context. Also, for brevity, by an  integer-valued polynomial,  we mean  a  polynomial in Int$(S,\Z)$  if the subset $S$ is clear from the context. 

\medskip

For a polynomial  $f  \in \mathrm{Int}(S, \Z)$, if $ p \mid f(a)\ \forall\ a \in S$, then $f$ has the following factorization  
$$f=p \tfrac{f}{p}$$ in the ring  $\mathrm{Int}(S, \Z).$  Hence, the given polynomial is trivially reducible.  
To get a criterion for the irreducibility of a polynomial $f  \in \mathrm{Int}(S, \Z)$   assumes that no prime $p$ divides  $f(a)\ \forall\ a \in S.$ Such a polynomial is called as {\em image primitive} polynomial. {\em By a polynomial in $ \mathrm{Int}(S, \Z)$, we mean an image primitive polynomial in $ \mathrm{Int}(S, \Z) $ unless specified otherwise..} We start this section with the following lemma.

\begin{lemma}\label{lem:first}  Let
 $f  \in \Z[x] $ be a  polynomial of degree $k$  and $p \in \Z$ be  a  prime number. Assume  $(b_i)_{i \geq 0}$ is a $p$-ordering of $S$ and $e$ is a positive integer. Then 
 $$ p^e \mid f(b_i)\  \forall\ 0 \leq i \leq k  \Leftrightarrow  p^e \mid f(b)\  \forall\  b \in S  .   $$

  \end{lemma}
 
   \begin{proof} We express the polynomial $f$ as
   
   \begin{equation}\label{eq:repre}
f=     \Sigma_{i=0}^k       c_i   (x-b_0)(x-b_1) \ldots (x-b_{i-1}).
   \end{equation}
   
   We claim that if $ p^e \mid f(b_i)\  \forall\ 0 \leq i \leq k ,$ then $ p^e \mid  c_i(b_i-b_0)(b_i-b_1) \ldots (b_i-b_{i-1})\  \forall\ 0 \leq i \leq k .$ We proceed   by induction. By Eq.(\ref{eq:repre}), $ p^e \mid f(b_0)$ gives us $ p^e \mid c_0$. Substituting $x=b_1$ in Eq.(\ref{eq:repre}), we get  $ p^e \mid c_0 +  c_1 (b_1-b_0) $.  As we know $ p^e \mid c_0$, it follows that $ p^e \mid   c_1 (b_1-b_0) $. Assume that the result is true till all the indices  $i-1 < k,$ i.e. $  p^e \mid c_j(b_j-b_0)(b_j-b_1) \ldots (b_j-b_{j-1})\ \forall\ 0 \leq j \leq  i-1.$ We substitute $x=b_i$ in Eq.(\ref{eq:repre}) and get the following
   
   \begin{equation}\label{eq:repre2}
f(b_i)=     \Sigma_{j=0}^i      c_j  (b_i-b_0)(b_i-b_1) \ldots (b_i-b_{j-1}).
   \end{equation}
  
 By induction hypothesis, $  p^e \mid c_j(b_j-b_0)(b_j-b_1) \ldots (b_j-b_{j-1})\ \forall\ 0 \leq j \leq  i-1$ and by the definition of  $p$-ordering $w_{p}((b_j-b_0)(b_j-b_1) \ldots (b_j-b_{j-1})) \mid w_{p}((b_i-b_0)(b_i-b_1) \ldots (b_i-b_{j-1}))\ \forall\ 0 \leq j \leq  i-1. $ Hence,   $ p^e$ must divide  $c_j(b_i-b_0)(b_i-b_1) \ldots (b_i-b_{j-1})\ \forall\ 0 \leq j \leq  i-1. $ By  Eq. (\ref{eq:repre2}) it follows that $ p^e$  divides $ c_{i}(b_{i}-b_0)(b_{i}-b_1) \ldots (b_{i}-b_{i-1})$. Since $i$ is arbitrary, our claim follows.
   
 To prove the statement of the lemma, we  substitute  $x=b$ in Eq.  (\ref{eq:repre2}) and obtain
 
$$ f(b)=     \Sigma_{j=0}^k      c_j  (b-b_0)(b-b_1) \ldots (b-b_{j-1}).$$

 Since $p^e$ divides $   f(b_i)\  \forall\ 0 \leq i \leq k ,$
 it must divide $c_i(b_i-b_0)(b_i-b_1) \ldots (b_i-b_{i-1})\  \forall\ 0 \leq i \leq k .$ As we know, $w_{p}((b_j-b_0)(b_j-b_1) \ldots (b_j-b_{j-1})) \mid w_{p}((b-b_0)(b-b_1) \ldots (b-b_{j-1}))\ \forall\ 0 \leq j \leq  k. $  Hence  $ p^e$   divides each term in the above expansion of $f(b)$ and consequently divides   $f(b)$.

 The converse part is trivial.
 
   

   \end{proof}

\medskip

Before proceeding further, we recall a fact which is important in the proof of Lemma \ref{ivpcrdseq}. For a given polynomial $f  \in \Z[x] $ and  a prime $p \in \Z$, let $a$ and $b$ be two integers such that $a \equiv b \pmod { p^e}, $
where $e$ is a positive integer. Then we must have  $f(a) \equiv f(b) \pmod { p^e}. $  Now we prove an important  lemma which is the cornerstone of this section.

 \begin{lemma}\label{ivpcrdseq}  For every polynomial $f=\tfrac{ g  }{d} \in \Q[x] $ of degree $k$, the following holds
     $$f \in \mathrm{Int}(S,\Z )   \Leftrightarrow f(a_i) \in \Z\ \forall\ 0 \leq i \leq k,$$
 where $a_0, a_1, \ldots, a_k$ is a $d_k$-ordering of $S \subseteq  \Z. $ \end{lemma}
 
   \begin{proof}

  Let $f(a_i)= \tfrac{g(a_i)}{d}\in \Z\ \forall\ 0 \leq i \leq k.$ Assume $p $  is  a prime dividing $ d$ such that $w_p(d)=p^e,$ then  $ p^e \mid g(a_i)\   \forall\ 0 \leq i \leq k.$ By our way of construction of $d_k$ ordering, there exists a $p$-ordering $(b_i)_{i \geq 0}$ such that $a_i \equiv b_i \pmod { p^{e_k}}, $ where $p^{e_k}=w_p(k!_S).$ Consequently,   $g(a_i) \equiv g(b_i) \pmod { p^{e_k}}. $   
   By Eq. (\ref{eq:factorial}) (or by Theorem 9 in Bhargava \cite{Bhar3} ), we must have $e \leq e_k. $ Since  it is given that $ p^e \mid g(a_i)\   \forall\ 0 \leq i \leq k,$  hence  $ p^e \mid g(b_i)\   \forall\ 0 \leq i \leq k.$  Invoking Lemma \ref{lem:first},  $ p^e$ divides $ g(b)\ \forall\ b \in S.$ Since this holds for each prime dividing $d$, it follows that $d$ divides $ g(b)\ \forall\ b \in S.$
  In other words,   $f = \tfrac{g }{d} \in  \mathrm{Int}(S,\Z  ) .$ 
  
  Conversely, let $f= \tfrac{g }{d} \in \mathrm{Int}(S,\Z  ).$ Assume $w_p(d)=p^e,$ then  $p^e$ must divide $g(b)\ \forall\ b \in S$. By the definition of $d_k$-ordering, for every $a_i$ there exists an element $b_i \in S$ such that $a_i \equiv b_i \pmod { p^{e_k}}, $  where $ 0 \leq i \leq k.$ As a result  $g(a_i) \equiv g(b_i) \pmod { p^{e_k}} $  and  $p^e$ must divide  $g(a_i)\ \forall\  0 \leq i \leq k.$ Since this holds for any prime $p$ dividing $d$; hence $d$ divides  $g(a_i)\ \forall\  0 \leq i \leq k.$ As a result,  $f(a_i)= \tfrac{g(a_i)}{d}\in \Z\ \forall\ 0 \leq i \leq k,$  which completes the proof.

   \end{proof}

 In other words, for a polynomial $f=\tfrac{ g(x)}{d} \in \Q[x] $ of degree $k$, a $d_k$-ordering is a test set to check whether $f$ is integer-valued or not.  The following lemma gives a criterion for image primitiveness of a given polynomial whose proof follows by Lemma \ref{ivpcrdseq}.

  \begin{lemma}\label{imprcrd}  A polynomial $f \in \mathrm{Int}(S,\Z )  $ of degree $k$ is image primitive iff $\forall\ p \mid k!_S,\ \exists\  0 \leq i \leq k  $ such that $  p \nmid f(a_i).$
  \end{lemma}

 The following lemma gives a criterion for a polynomial to be integer-valued.
 
  \begin{lemma}\label{ivpcr}  A polynomial $f=\tfrac{ \sum_{i=0}^{k} b_iF_i(x)}{d} \in \Q[x] $ is integer-valued   over $S$ iff $ \forall\ p \mid d, w_p(d) \mid w_p(b_ii!_S)\ \forall\  0 \leq i \leq k. $
  \end{lemma}
 
 \begin{proof} By Lemma \ref{ivpcrdseq}, a given polynomial $f=\tfrac{ \sum_{i=0}^{k} b_iF_i(x)}{d} \in \Q[x] $ is integer-valued over $S$   iff
 
 $$f(a_i) =\tfrac{ g(a_i)}{d} \in \Z\ \forall\ 0 \leq i \leq k.$$

Now we proceed by induction. We assume that the polynomial is integer-valued over $S$. Substituting $x=a_0,$ we get $ d$ divides $b_00!_S,$ where $0!_S=1$. Assume that the result holds for all indices up to $i-1<k$. We put $x=a_i$ and obtain

$$f(a_i)=\tfrac{ \sum_{j=0}^{i} b_jF_j(a_i)}{d} \in \Z. $$

By induction hypothesis, $w_p(d)$ divides $b_jj!_S\ \forall\ 0 \leq j \leq i-1$ for all $p \mid d$, hence divides $
w_{p}(b_j(a_i-a_0)(a_i-a_1) \ldots (a_i-a_{j-1}))\ \forall\ 0 \leq j \leq  i-1. $ Since $w_p(d)$ divides $g(a_i),$ it must divide  $w_p( b_i(a_i-a_0)(a_i-a_1) \ldots (a_i-a_{i-1})) $ for all $p \mid d$.
By our way of construction of $d_k$-ordering (or by Lemma  \ref{conlem}), $w_p( (a_i-a_0)(a_i-a_1) \ldots (a_i-a_{i-1}))  $ is same as $w_p(i!_S)$ for all $p \mid d$, hence $w_p(d) \mid b_jj!_S\ \forall\ 0 \leq j \leq k    $ for all $p \mid d.$

Conversely, assume $w_p(d) \mid b_jj!_S\ \forall\ 0 \leq j \leq k    $ for all $p \mid d.$ For any $0 \leq i \leq k$, $f(a_i)$ can be written as

\begin{equation}\label{eq:above}
 f(a_i)=\tfrac{ \sum_{j=0}^{i} b_jF_j(a_i)}{d}.  
\end{equation}

Since for every  $p \mid d,\ w_p(j!_S) \mid w_p( (a_i-a_0)(a_i-a_1) \ldots (a_i-a_{j-1}), $ it follows that $w_p(d)$ divides each term in the numerator of Eq. (\ref{eq:above}) by Lemma \ref{conlem}. 
 Consequently, $f(a_i) \in \Z\ \forall\ 0 \leq i \leq k  $ 
and the polynomial is integer-valued over $S$ by Lemma \ref{ivpcrdseq}.


 \end{proof}
 
  The following lemma also gives a criterion for image primitiveness. Before the lemma, recall that (see \cite{Bhar1}, \cite{Bhar2}, \cite{Bhar3} etc), for any polynomial  $f = \tfrac{ g }{d}\in \mathrm{Int}(S,\Z )  $ of degree $k$ and any   prime $p$, we always have $ w_p(d )   \leq w_p(d(S,g)) \leq w_p(k!_S).$  
 
  \begin{lemma}\label{imprcr}  A polynomial $f = \tfrac{ \sum_{i=0}^{k} b_iF_i(x) }{d}\in \mathrm{Int}(S,\Z )  $ is image primitive iff $\forall\ p \mid k!_S,\ \exists\  0 \leq i \leq k  $ such that $  w_p(d) =w_p(b_ii!_S).$
  \end{lemma}

  \begin{proof} Let the polynomial  $f = \tfrac{ \sum_{i=0}^{k} b_iF_i(x) }{d}\in \mathrm{Int}(S,\Z )  $ be image primitive. By Lemma \ref{ivpcr}, for any prime  $p \mid d$ we must have  $w_p(d) \leq w_p(b_ii!_S)\ \forall\ 0\leq i\leq k$. If there exists some prime $p\mid d$ such that  $  w_p(d) < w_p(b_ii!_S)\ \forall\ 0\leq i\leq k,$ then $  w_p(pd) \leq  w_p(b_ii!_S)\ \forall\ 0\leq i\leq k.$ Again by Lemma \ref{ivpcr}, the polynomial 
$\tfrac{ f}{p} = \tfrac{ \sum_{i=0}^{k} b_iF_i(x) }{pd}\in \mathrm{Int}(S,\Z ).  $ In other words $p \mid f(a)\ \forall\ a \in S$ and $f$ is not image primitive. Which is a contradiction. Hence there must exist a $ 0\leq i\leq k $  such that $  w_p(d) = w_p(b_ii!_S)\ \forall\ p \mid k!_S
.$ A fortiori this condition holds for all $p\mid k!_S$, since by the inequality before Lemma 3.5, $p\mid d\Rightarrow p\mid k!_S$ (if $p\nmid d,$ then the condition obviously holds).
  
  Conversely, assume that for every $  p \mid k!_S,\ \exists\  0 \leq p_i \leq k,  $ such that $  w_p(d) =w_p(b_{p_i}p_i!_S)   $ and the polynomial $f$ is not image primitive. Hence, there exists a prime $  p \mid k!_S$ such that $  p \mid f(a)\ \forall\  a  \in S.$ Consequently, the polynomial $   \tfrac{ f}{p}  = \tfrac{ \sum_{i=0}^{k} b_iF_i(x) }{pd}$ is a member of $ \mathrm{Int}(S,\Z ) . $ By Lemma \ref{ivpcr}, we get  $  w_p(pd)  \leq w_p(b_{p_i}p_i!_S)$ or  $  w_p(d) < w_p(b_{p_i}p_i!_S)$ which is a contradiction. Hence, $f$ must be image primitive.
  .  
    \end{proof}

From this lemma, it follows that a given polynomial $f = \tfrac{ g }{d}\in \mathrm{Int}(S,\Z )  $ is image primitive iff $d=d(S,g).$       Observe that if a polynomial  $f   \in \mathrm{Int}(S,\Z ), $ which factors as $f_1f_2$ in $\mathrm{Int}(S,\Z ) $, is image primitive, then both the polynomials  $f_1$ and $f_2$ must be image primitive.  
Now we   prove  our main theorem.

     \begin{theorem}\label{mainth}  Let $f = \tfrac{g}{d} \in \mathrm{Int}(S,\Z ) $ be  a polynomial  of degree $k$ and $a_0, a_1, \ldots, a_k$ be a $d_k$-ordering. Then $f$ is irreducible iff    
for any factorization $g=
( \sum_{i=0}^{k_1} b_iF_i(x) )$ $( \sum_{j=0}^{k_2} c_jF_j(x) )$   there exist a prime $p \mid d$ and  
non-zero positive integers $r  \leq k_1$ and $s  \leq k_2$ 
such that  
$  \tfrac{    \mu_r (d , p)\mu_s (d , p) }    {w_p(d)}    \nmid w_{p}(b_rc_s ).$
  
  \end{theorem}
  
  \begin{proof}
  For the given  polynomial   $f = \tfrac{g}{d} \in \mathrm{Int}(S, \Z  )   $, suppose that
      for   every factorization   $g=
( \sum_{i=0}^{k_1} b_iF_i(x) )( \sum_{j=0}^{k_2} c_jF_j(x) ),$      
 there exist a prime  $p \mid d  $ and a pair $(r,s) \in \W^2$ such that  $  \tfrac{    \mu_r (d , p)\mu_s (d , p) }    {w_p(d)}    \nmid w_{p}(b_rc_s ).$ 
  Let the contrary be true. 
 Then we can write 
 
 $$f= \dfrac{g_1}{d_1}  \dfrac{g_2}{d_2}= \dfrac{ \sum_{i=0}^{k_1} b_iF_i(x) }{d_1}  \dfrac{\sum_{j=0}^{k_2} c_jF_j(x)}{d_2}, $$
 
 such that both of the polynomials, $\tfrac{ \sum_{i=0}^{k_1} b_iF_i(x) }{d_1}$ and $  \tfrac{\sum_{j=0}^{k_2} c_jF_j(x)}{d_2} $  
 are integer-valued.
 
 Since $\tfrac{ \sum_{i=0}^{k_1} b_iF_i(x) }{d_1}$ and $  \tfrac{\sum_{j=0}^{k_2} c_jF_j(x)}{d_2} $  are  
 image primitive, by Lemma \ref{imprcr}, for each prime$p \mid d$, there exists integers   $0 \leq r \leq k_1$ and $0 \leq s \leq k_2 $ depending on $p,$
   such that 
 $w_p(d_1)=  w_p(b_rr!_S)$ and    $w_p(d_2)=  w_p(c_ss!_S).$ 
 Multiplying and rearranging we get $ \tfrac{ w_p(d_1)w_p(d_2)}{    w_p(r!_S) w_p(s!_S)}    = w_{p}(b_rc_s ).$ 
 Using the definition of  $\mu_i (d , p)$ where $i \in \W$, we get $  \tfrac{    \mu_r (d , p)\mu_s (d , p) }    {w_p(d)}    = w_{p}(b_rc_s ) $,   which is a contradiction.

Conversely, let  $f = \tfrac{g}{d} \in \mathrm{Int}(S,\Z  ) $ be irreducible and   $g=g_1g_2= ( \sum_{i=0}^{k_1} b_iF_i(x) )$ $ ( \sum_{j=0}^{k_2} c_jF_j(x) )$ be a factorization. 
 We can  factor $d$ as $d_1 d_2$ in such a way that in the factorization 
 $$f= \dfrac{g_1}{d_1}  \dfrac{g_2}{d_2}, $$
 
 $d_1$ is the biggest integer such that 
 the polynomial $  \tfrac{g_1}{d_1}$ is  integer-valued. The polynomial  $    \tfrac{g_2}{d_2}$ cannot be integer-valued otherwise $f$  would be reducible. We have  the following observations (see Lemma  \ref{imprcr} and  Lemma  \ref{ivpcr})
 
 $  \tfrac{g_2}{d_2} \notin \mathrm{Int}(S,\Z ) \Rightarrow  \exists$ \ a\ prime\ $ p \mid d  $ and an integer $0 \leq s \leq \deg(g_2)$ such that $w_p(d_2) > w_p(c_ss!_S),$

and 
 $  \tfrac{g_1}{d_1} \in \mathrm{Int}(S,\Z ) \Rightarrow  \forall$ prime $q \mid d_1$ there exists a positive integer $0 \leq r \leq \deg(g_1)$ such have $w_q(d_1)= w_q(b_rr!_S).$

 Combining both of these observations together, we conclude that there exist a prime $p \mid d$ and a pair $ (r,s) \in \W^2$ such that    $w_p(d_1) \geq w_p(b_rr!_S)$ and    $w_p(d_2)>  w_p(c_ss!_S).$ Now we proceed as in the previous part to obtain 
 
  $$  \tfrac{    \mu_r (d , p)\mu_s (d , p) }    {w_p(d)}    > w_{p}(b_rc_s ), $$
 
 which completes the proof.
 
 
  \end{proof}

 Now we give some examples to explain Theorem \ref{mainth}. 
 
 \begin{Ex} Let us check the irreducibility of the polynomial  $$f=\tfrac{18x^6-48x^5+47x^4-29x^2+41x+6}{6} $$
 in $\mathrm{Int}(\Z).$
  We have only the following way of factoring $f$
   
 $$f = \tfrac{f_1f_2      }{6},
 $$

 where $f_1= 2+3x+6x(x-1)+2x(x-1)(x-2)$ and $f_2=  3+4x+3x(x-1)+9x(x-1)(x-2) .$ Now we carefully inspect the coefficients of the polynomials. It is obvious that both $b_0=2$ and $c_1=4$ are not multiple of three. It means $ w_3(b_0c_1)=w_3(2 \times4)=3^0.$  However,  $\tfrac{    \mu_0(6 ,3)\mu_1 (6 , 3) }    {w_3(6)} = \tfrac{3^13^1}{3^1}   > w_{p}(b_rc_s ), $ which implies that the polynomial is  irreducible.
 \end{Ex}

 \begin{Ex} Let us check the irreducibility of the polynomial  $$f=\tfrac{2x^6+9x^5-38x^4-21x^3+57x^2-42x+9}{9} $$
 in $\mathrm{Int}(\Z).$
 Here, we have  the following   factorization
   
 $$f = \tfrac{f_1f_2      }{6},
 $$

 where $f_1= 3+x+9x(x-1)+x(x-1)(x-2)$ and $f_2=  3-9x+3x(x-1)+2x(x-1)(x-2) .$  Here $b_1=1$ and $c_3=2$. The value of  $\tfrac{    \mu_1(9 , 3)\mu_3 (9 , 3) }    {w_3(9)} $ is  $\tfrac{3^23^1}{3^2} $ which does not divide  $ w_3(b_1c_3)=w_3(1   \times2)=3^0.$ Hence, the polynomial is irreducible.
 \end{Ex}
 \medskip
 
By the examples above it is obvious that sometimes Theorem \ref{mainth} may become extremely powerful in testing the irreducibility of a polynomial when its factor's representation in the falling factorial (or generalized falling factorial) basis is known. Also,  by  our approach   all the results in this section can be generalized for any arbitrary subset of a Dedekind domain.  For the sake of completeness, we state the analog of  Theorem  \ref{mainth} in the case of an arbitrary subset of a Dedekind domain.

\begin{theorem}\label{mainth1}  Let $f = \tfrac{g}{d} \in \mathrm{Int}(S,R ) $ be  a polynomial  of degree $k$ and $a_0, a_1, \ldots, a_k$ be a $d_k$-ordering of $S$. Then $f$ is irreducible iff    
for any factorization $g=
( \sum_{i=0}^{k_1} b_iF_i(x) )$ $( \sum_{j=0}^{k_2} c_iF_i(x) ),$   there exists a prime ideal $P \mid d$ and  
non-zero positive integers $r  \leq k_1$ and $s  \leq k_2$ 
such that  
$  \tfrac{    \mu_r (d , P)\mu_s (d , P) }    {w_P(d)}    \nmid w_{P}(b_rc_s ).$
  
  \end{theorem}

\section*{Acknowledgement} We thank the anonymous referee for his valuable suggestions, which really improved the paper. The authtor wishes to thank Professor  Louiza Fouli (Associate Editor, RMJM) for her help.

\end{document}